\documentclass[11pt]{amsart}
\usepackage{amssymb}
\usepackage{verbatim}
\usepackage{hyperref}
\usepackage{pifont}
\usepackage{color}

\begin{document}

\newtheorem{theorem}{Theorem}    
\newtheorem{proposition}[theorem]{Proposition}
\newtheorem{conjecture}[theorem]{Conjecture}
\def\theconjecture{\unskip}
\newtheorem{corollary}[theorem]{Corollary}
\newtheorem{lemma}[theorem]{Lemma}
\newtheorem{sublemma}[theorem]{Sublemma}
\newtheorem{observation}[theorem]{Observation}
\theoremstyle{definition}
\newtheorem{definition}{Definition}
\newtheorem{notation}[definition]{Notation}
\newtheorem{remark}[definition]{Remark}
\newtheorem{question}[definition]{Question}
\newtheorem{questions}[definition]{Questions}
\newtheorem{example}[definition]{Example}
\newtheorem{problem}[definition]{Problem}
\newtheorem{exercise}[definition]{Exercise}

\numberwithin{theorem}{section}
\numberwithin{definition}{section}
\numberwithin{equation}{section}

\def\earrow{{\mathbf e}}
\def\rarrow{{\mathbf r}}
\def\uarrow{{\mathbf u}}
\def\tpar{T_{\rm par}}
\def\apar{A_{\rm par}}

\def\reals{{\mathbb R}}
\def\torus{{\mathbb T}}
\def\heis{{\mathbb H}}
\def\integers{{\mathbb Z}}
\def\rationals{{\mathbb Q}}
\def\naturals{{\mathbb N}}
\def\complex{{\mathbb C}\/}
\def\distance{\operatorname{distance}\,}
\def\support{\operatorname{support}\,}
\def\dist{\operatorname{dist}\,}
\def\Span{\operatorname{span}\,}
\def\degree{\operatorname{degree}\,}
\def\kernel{\operatorname{kernel}\,}
\def\dim{\operatorname{dim}\,}
\def\codim{\operatorname{codim}}
\def\trace{\operatorname{trace\,}}
\def\Span{\operatorname{span}\,}
\def\dimension{\operatorname{dimension}\,}
\def\kernel{\operatorname{Ker}}
\def\p{\partial}
\def\rp{{ ^{-1} }}
\def\Re{\operatorname{Re\,} }
\def\Im{\operatorname{Im\,} }
\def\ov{\overline}
\def\eps{\varepsilon}
\def\lt{L^2}
\def\diver{\operatorname{div}}
\def\curl{\operatorname{curl}}
\newcommand{\norm}[1]{ \|  #1 \|}
\def\expect{\mathbb E}
\def\variance{\operatorname{Var}}
\def\prob{\operatorname{Pr}}
\def\bull{$\bullet$\ }

\newcommand{\Norm}[1]{ \left\|  #1 \right\| }
\newcommand{\set}[1]{ \left\{ #1 \right\} }
\def\one{\mathbf 1}
\newcommand{\modulo}[2]{[#1]_{#2}}

\def\scriptf{{\mathcal F}}
\def\scriptg{{\mathcal G}}
\def\scriptm{{\mathcal M}}
\def\scriptb{{\mathcal B}}
\def\scriptc{{\mathcal C}}
\def\scriptt{{\mathcal T}}
\def\scripts{{\mathcal S}}
\def\scripti{{\mathcal I}}
\def\scripte{{\mathcal E}}
\def\scriptv{{\mathcal V}}
\def\scriptS{{\mathcal S}}
\def\scripta{{\mathcal A}}
\def\scriptr{{\mathcal R}}
\def\scripto{{\mathcal O}}
\def\scripth{{\mathcal H}}
\def\scriptd{{\mathcal D}}
\def\scriptl{{\mathcal L}}
\def\scriptn{{\mathcal N}}
\def\frakv{{\mathfrak V}}

\author{Michael Christ}
\address{
        Michael Christ\\
        Department of Mathematics\\
        University of California \\
        Berkeley, CA 94720-3840, USA}
\email{mchrist@math.berkeley.edu}
\thanks{The author was supported in part by NSF grant DMS-0901569.}

\date{March 9, 2010. Revised August 18, 2011.}

\title
{On Random Multilinear Operator Inequalities} 


\maketitle

\section{Introduction}
A venerable principle 
holds that the Fourier transform of a measure is ``small'' in a meaningful
sense when linear structure is absent, 
in certain circumstances. 
For instance: 
\begin{enumerate}
\item
If $\mu$ is supported
on an appropriately curved submanifold of  $\reals^d$,
then $\widehat{\mu}(\xi)\to 0$ at a certain rate as $|\xi|\to\infty$.
\item
If $\mu_\omega$ is a random measure, with appropriate properties, 
then for typical $\omega$, $\widehat{\mu_\omega}$ has
small supremum norm; or in other contexts, 
$\widehat{\mu_\omega}(\xi)$ tends to zero at an appropriate rate as $|\xi|\to\infty$.
\item
Let $p$ be a large prime, and for $x\in\integers_p$ let $\mu_p(x)=1$ if
$x$ is a quadratic residue modulo $p$, and $\mu_p(x)=0$ otherwise.
Then with a natural normalization of the Fourier transform, 
$|\widehat{\mu_p}(\xi)|\le Cp^{-1/2}$ for all $\xi\ne 0$,
whereas $\widehat\mu_p(0)\asymp 1$.
\end{enumerate}

Smallness of the Fourier transform may be reformulated in terms of a bilinear expression
via the identity 
$\norm{\widehat{\mu}}_\infty = \sup_{f,g\ne 0} \big|\iint f(x)g(y)\,d\mu(x-y)\big|/\norm{f}_2\norm{g}_2$. 
This formulation suggests multilinear extensions, involving e.g.\ 
$\iint f(x)g(y)h(x+y)d\mu(x-y)$.
While various possible inequalities can be considered, we are primarily 
interested in bounds in terms of $\norm{f}_p\norm{g}_q\norm{h}_r$ 
with $p^{-1}+q^{-1}+r^{-1}=1$;
such quantities scale naturally from the perspective of ergodic theory.

If $G$ is a finite Abelian group and $\mu:G\to\complex$ has $\norm{\widehat{\mu}}_\infty\ll \norm{\mu}_1$,
under appropriate normalizations, then $\mu$ is sometimes said to be {\em uniform} \cite{taovu}.
There are higher-order notions of uniformity, due to Gowers \cite{taovu}, which have a multilinear
character. Gowers uniformity is closely related to the type of smallness studied in this
paper, but here we are dealing with rather singular measures.

In this paper we investigate the extension of this smallness principle to higher-degree multilinear expressions,
for natural families of random measures.
In \S\ref{section:example} we give an example which demonstrates that
linear structure is no longer the natural consideration. Indeed,
for one of the most canonical (deterministic) examples of all, the natural trilinear extension satisfies no 
smallness condition, due to the presence of {\em quadratic} structure. 
In \S\ref{section:results} we state our main results, which concern two classes of random measures.
For one of these classes, our results are quite satisfactory, but for the other they 
apply only for a certain range of parameters which may not be optimal.

The author is indebted to Patrick LaVictoire for useful discussions.

\section{A Nonlinear Obstruction} \label{section:example}

For convenience, the following example is given in the context of certain finite groups,
rather than $\integers$; there are no essential differences.
Let $d\ge 1$ and let $p\in\naturals$ be any prime.
Let $\integers_p$ be the finite cyclic group $\integers/ p\integers$.
Let $G_p = \integers_p^d\times\integers_p$.
For $x=(x_1,\cdots,x_{d})\in \integers_p^d$, we write 
$|x|^2 = \sum_{j=1}^d x_j^2$.
Write $G_p\owns x=(x',x_{d+1})\in \integers_p^d\times\integers_p$.
Let $\mu_p$ be the function on $G_p$ defined by
\[ \mu_p(x',x_{d+1})=
\begin{cases} p^{-d} &\text{ if } x_{d+1}=|x'|^2 \\
0 &\text{ otherwise}
\end{cases},
\]
\[ m_p(x)=
p^{-d-1} 
\ \text{for all } x\in G_p,
\]
and $\nu_p=\mu_p-m_p$.
$m_p$ satisfies 
$\widehat{m_p}(\xi)=\widehat{\mu_p}(\xi)$ for $\xi=0$,
and 
$\widehat{m_p}(\xi)=0$ for all $\xi\ne 0$.

Define the Fourier transform
$\widehat{f}(\xi) = 
\sum_{x\in G_p} f(x)e^{-2\pi i \xi\cdot x/p}$,
where $\xi\cdot x = \sum_{j=1}^{d+1}x_j\xi_j$
and $\xi\in G_p$.
By a well-known identity for Gauss sums, 
\[
\max_{\xi\ne 0} |\widehat{\mu_p}(\xi)|= p^{-d/2},
\]
and consequently
\[
\max_\xi |\widehat{\nu_p}(\xi)|= p^{-d/2}.
\]
Therefore by Plancherel's identity, there is a bilinear inequality
\begin{equation} \label{bilinearexample}
\big|
\sum_{x,y\in G_p} f(x)g(y)\nu_p(x-y)
\big|
\le p^{-d/2} \norm{f}_2 \norm{g}_2
\end{equation}
where $\norm{f}_q$ denotes the $\ell^q(G_p)$ norm.

Does \eqref{bilinearexample} extend to a trilinear inequality 
\begin{equation} \label{eq:nope}
\big|
\sum_{x,y\in G_p} f(x)g(y)h(x+y)\nu_p(x-y)
\big|
\le Cp^{-\rho} \norm{f}_2 \norm{g}_2\norm{h}_\infty
\end{equation}
for some $\rho=\rho(d)>0$ independent of $p$?

\begin{observation}
No inequality of the form \eqref{eq:nope} is valid.
\end{observation}

To prove this, set 
\begin{align*}
h(x)&=e^{2\pi i |x'|^2/p}
\\
f(x) &= e^{2\pi i [x_{d+1}-2|x'|^2]/p}
\\
g(x) &= e^{2\pi i [-x_{d+1}-2|x'|^2]/p}.
\end{align*}
Then
$f(x)g(y)h(x+y)=e^{2\pi i\Phi(x,y)/p}$ where
\[
\Phi(x,y) = 
|x'+y'|^2 + x_{d+1}-y_{d+1}-2|x'|^2-2|y'|^2.
\]
For $(x,y)$ in the support of $\mu_p(x-y)$,
$x_{d+1}-y_{d+1}\equiv |x'-y'|^2$ and consequently
\[
\Phi(x,y) = |x'+y'|^2 + |x'-y'|^2-2|x'|^2-2|y'|^2
\equiv 0.
\]
Therefore the contribution of $\mu_p$ to our trilinear form equals
\[
\sum_{x,y\in G_p} f(x)g(y)h(x+y)\mu_p(x-y)
= 
\sum_{x,y: y_{d+1}-x_{d+1}=|y'-x'|^2} p^{-d}
= p^{2d+1}p^{-d}
= p^{d+1},
\]
while
\[
\norm{f}_2 \norm{g}_2\norm{h}_\infty
= (p^{d+1})^{1/2}
\cdot  (p^{d+1})^{1/2}
\cdot 1
= p^{d+1}.
\]
On the other hand,
\[
\sum_{x,y\in G_p} f(x)g(y)h(x+y) m_p(x-y)
= 
p^{-d-1} \sum_{x,y\in G_p} e^{2\pi i \Phi(x,y)/p}.
\]
For fixed $y,x'$, $\Phi((x',t),y)$ takes the form $c(x',y)+t$,
and 
\[\sum_{t\in \integers_p} e^{2\pi i (t+c(x',y))/p}
=e^{2\pi i c(x',y)/p}\sum_{t\in \integers_p} e^{2\pi i t/p} \equiv 0.\]
Thus in all,
\[
\sum_{x,y\in G_p} f(x)g(y)h(x+y)\nu_p(x-y)
=
\sum_{x,y\in G_p} f(x)g(y)h(x+y)\mu_p(x-y)
= 
\norm{f}_2 \norm{g}_2\norm{h}_\infty;
\]
there is no cancellation in the sum.

\section{Results} \label{section:results}

Our setting is the set $\integers$ of all integers, and
we will work in terms of norms $L^p(\integers)=\ell^p$.
Let $\Omega$ be a probability space, equipped with
jointly independent, identically distributed,
$\{0,1\}$--valued selector variables $\{s(\omega,x): x\in\integers\}$,
such that $s(\omega,x)=1$ with probability $p$
and $=0$ with probability $1-p$.
Let $N$ be any large positive integer.
Let $r(\omega,x)
= (Np)^{-1}s(\omega,x) -N^{-1}$
for integers $x\in[-N,N]$,
and $r(\omega,x)=0$ otherwise.
Thus $\expect_\omega r(\omega,x)
= p(Np)^{-1}-N^{-1}=0$ for $x\in[-N,N]$.

Let $\{L_j: 0\le j\le M\}$ be $\integers$-linear mappings
from $\integers$ to $\integers$.
Assume none of the
$L_j$ are scalar multiples, over $\rationals$, of $(x,y)\mapsto x$,
that none are scalar multiples of $(x,y)\mapsto y$, 
and no $L_i$ is a scalar multiple of $L_j$.

In Theorem~\ref{thm:main}
we study multilinear operators
\begin{equation} \label{Tomega}
T_\omega(f,g_1,\cdots,g_M)(x)
= \sum_y f(y)r(\omega,L_0(x,y))\prod_{j=1}^M g_j(L_j(x,y)).
\end{equation}
These depend also on $N$, and we are interested in their properties as $N\to\infty$.
Define the operator norm
\[ 
\norm{T_\omega}_{\text{op}}= \sup_{f,g_1,\cdots,g_M}\norm{T_\omega(f,g_1,\cdots,g_M)}_2
\]
where 
the supremum is taken over all functions satisfying $\norm{f}_2\le 1$
and $\norm{g_j}_\infty\le 1$ for all $j$.

\begin{theorem} \label{thm:main}
Suppose that $M\ge 1$ 
and $0\le\gamma<2^{-M}$.
There exist $\eps>0$ and $C<\infty$
such that 
for all $N\ge 1$ and $p\ge N^{-\gamma}$,
\begin{equation}
\expect_\omega\norm{T_\omega}_{\text{op}}
\le CN^{-\eps}.
\end{equation}
\end{theorem}
\noindent
The constant $C$ is independent of $N$.
We do not know whether the conclusion may hold for a 
larger range of exponents $\gamma$.

Of course
\[\expect_\omega
\sup_{f,g_1,\cdots,g_M}\norm{T_\omega(f,g_1,\cdots,g_M)}_2 \ge
\sup_{g_1,\cdots,g_M}\expect_\omega
\sup_f\norm{T_\omega(f,g_1,\cdots,g_M)}_2.\]
The latter quantity is easier to analyze; see Proposition~\ref{prop:easier},
which gives a satisfactory bound for all $\gamma<1$, for all $M$.

An ergodic-theoretic consequence is as follows.
Let $T$ be an invertible measure-preserving transformation
on a probability space $(X,\mu)$.
For each $\omega\in\Omega$,
specify a subsequence $(n_k(\omega): k=1,2,\cdots)$ of the natural numbers,
as follows. 
Let $\gamma\in (0,1)$. 
Let $\Omega$ be a probability space equipped with a family
of jointly independent random variables $\{s_n(\omega): n\in\naturals\}$
such that $s_n(\omega)=1$ with probability $n^{-\gamma}$,
and $s_n(\omega)=0$ otherwise. 
For each $\omega\in\Omega$, specify the random subsequence $(n_k(\omega))_{k\in\naturals}$ to
consist of all $n\in\naturals$ for which $s_n(\omega)=1$, listed in increasing order.

It has been proved \cite{austinpleasant1},\cite{hostkra},\cite{tao} that
for all $f_1,\cdots,f_M\in L^\infty(X)$,
\begin{equation} \label{eq:RTknown}
\lim_{N\to\infty} N^{-1}\sum_{k=1}^N f_1(T^{k}(x)) f_2(T^{2k}(x)) \cdots
f_M(T^{Mk}(x)) \text{ exists in } L^1(X,d\mu(x)).  \end{equation}
This fundamental result, together with Theorem~\ref{thm:main}, give
\begin{theorem} \label{thm:multimeanconvergence}
If $0\le \gamma<2^{-M+1}$ then
for almost every $\omega\in\Omega$,
for all $f_1,\cdots,f_M\in L^\infty(X)$,
\begin{equation}
\lim_{N\to\infty} N^{-1}\sum_{k=1}^N f_1(T^{n_k}(x)) f_2(T^{2n_k}(x)) \cdots
f_M(T^{Mn_k}(x)) \text{ exists in } L^1(X,d\mu(x)).
\end{equation}
\end{theorem}


A generalization of Theorem~\ref{thm:main} is natural, and of interest.
$e_\xi$ will denote the function $y\mapsto e^{-i\xi y}$.
With the above notations, define
\begin{align*}
T_{\omega}^*(f,g_1,\cdots,g_M)(x)
&= \sup_{\xi\in\torus} \Big|\sum_y e^{-i\xi y} f(y)r(\omega,L_0(x,y))\prod_{j=1}^M g_j(L_j(x,y))\Big|
\\
&= \sup_{\xi\in\torus} \big|T_{\omega}(e_{\xi} f,g_1,\cdots,g_M)(x)\big|.
\end{align*}
Multiplying each function $g_j(z)$ by a factor $e^{-i\xi_j z}$,
and taking the supremum over all $(\xi,\xi_1,\cdots,\xi_M)$,
introduces no additional generality since 
each $e^{-i\xi_j L_j(x,y)}$ can be factored as $e^{ia_j x\xi_j}e^{-ib_j y\xi_j}$
for appropriate coefficients $a_j,b_j$.

\begin{theorem} \label{thm:Carleson}
For each $0\le\gamma<2^{-M-1}$
there exist
$\delta>0$ and  $C<\infty$
such that for all $N\ge 1$ and $p\ge N^{-\gamma}$,
\begin{equation}
\expect_\omega \norm{T_{\omega}^*}_{\text{op}}
\le CN^{-\delta}.
\end{equation}
\end{theorem}

The case $M=0$ has an ergodic-theoretic consequence,
for return times of sparse random subsequences.
\begin{theorem}[Return Times] \label{thm:returntimes}
Let $(X,\scripta,\mu,\tau)$ be any dynamical system,
such that $\mu$ is a probability measure
and $(X,\scripta,\mu)$ is isomorphic to $[0,1]$
equipped with Lebesgue measure and the Lebesgue $\sigma$-algebra.
Let $0\le\gamma<\tfrac12$.
Let $\{n_k(\omega)\}$ be a random sequence, constructed as in Theorem~\ref{thm:multimeanconvergence}.
Let $p\in(1,\infty]$ and $q\ge 2$.
Then for almost every $\omega\in\Omega$, the following holds.
For each $f\in L^p(X)$
there exists a subset $X_0\subset X$ of full measure such that
for every dynamical system $(Y,\scriptf,\nu,\sigma)$,
every $g\in L^q(Y)$, and every $x\in X_0$,
\begin{equation*}
\lim_{N\to\infty}N^{-1}\sum_{k=1}^N f(\tau^{n_k(\omega)}(x))g(\sigma^{n_k(\omega)}(y))
\text{ exist for  $\nu$-almost every $y\in Y$.}
\end{equation*}
\end{theorem}

Thus far we have considered random variables which depend only on $L(x,y)$
for some linear function $L$. Next, we consider analogous results
for random matrices $\begin{pmatrix} r_\omega(x,y)\end{pmatrix}_{x,y}$, with all entries mutually independent.
Consider jointly independent random selector variables $s_\omega(x,y)$
for $(x,y)\in[-N,\cdots,N]^2$, satisfying $s_\omega(x,y)=1$ with probability
$p$, and $=0$ otherwise.
Then $\expect(\sum_x s_\omega(x,y))\asymp Np$
and $\expect(\sum_y s_\omega(x,y))\asymp Np$.
Define $r_\omega(x,y)=(Np)^{-1}\big(s_\omega(x,y)-p)$
so that
$\expect_\omega r_\omega(x,y) =0$.

\begin{theorem} \label{thm:morerandom}
Let $M\ge 2$ and $0\le\gamma<1$.
For any $\{L_j: 0\le j\le M\}$ satisfying
the hypotheses of Theorem~\ref{thm:main}
and for any $\eps>0$ there exists $C_{M,\eps}<\infty$
such that for all $N\ge 1$
and all $p\ge N^{-\gamma}$,
the multilinear forms 
\[
\scriptt_\omega(f_1,\cdots,f_M) = \sum_{x,y} r_\omega(x,y)\prod_{j=1}^M f_j(L_j(x,y))
\]
satisfy
\begin{equation}
\expect_\omega\norm{\scriptt_\omega}_{\text{op}}
\le C_{M,\eps}N^\eps N^{-(1-\gamma)/2}.
\end{equation}
\end{theorem}

In this formulation, $\scriptt_\omega(f_1,\cdots,f_M)$
is a complex number, not a function.
It is possible to generalize Theorem~\ref{thm:morerandom} by incorporating factors
$e^{-i\xi y}$, with a supremum over all $\xi$,
parallel to Theorem~\ref{thm:Carleson}.

The conclusion of Theorem~\ref{thm:morerandom} fails to hold for $\gamma>1$.
The method of proof of Theorem~\ref{thm:main} applies only
in the restricted range $\gamma< 2^{-(M-2)}$,
and with some added complications since
the Fourier transform cannot be applied directly.
However, our proof for the full range $\gamma<1$ proceeds along quite different lines,
relying on entropy considerations along with large deviations bounds.

\section{A Preliminary Bound}

The order of quantifiers in Theorem~\ref{thm:main} is significant.
In this preliminary section we discuss a variant in which the supremum 
in the definition \eqref{Tomega} of $T_\omega$
is taken only over $f$, with $g_1,\cdots,g_M$ fixed.
For this variant, and even for a substantial generalization, 
more complete results can be obtained, by a simpler method.

Generalize $T_\omega$ by considering linear operators 
\begin{equation}
L_{\omega,h}(f)(x)=\sum_y f(y) r(\omega,x-y)h(x,y),
\end{equation}
where $h\in \ell^\infty(\integers^2)$ is an arbitrary bounded
function of two variables.
In particular, this includes the case where
$h(x,y)=\prod_{j=1}^M g_j(L_j(x,y))$, for
arbitrary $M$ and  $g_j\in\ell^\infty$. 

Let $\Omega$, $p$, $N$, $s(\omega,\cdot)$, $r(\omega,x)$
be as in Theorem~\ref{thm:main}.
Regard $L_{\omega,h}$ as a linear opertor on $\lt([-N,N])$.
\begin{proposition} \label{prop:easier}
For any $\eps>0$,
there exists $C_\eps<\infty$ such that
for every $h\in\ell^\infty$,
\begin{equation}
\expect_\omega
\norm{L_{\omega,h}}_{\text{op}}
\le
C_\eps
N^\eps (Np)^{-1/2}
\norm{h}_{\ell^\infty}.
\end{equation}
\end{proposition}

\begin{proof}
Fix $N$.
Denote by $\trace$ the trace of a self-adjoint linear operator
on $\ell^2([-N,N])$.
Fix $h$, and write $L_\omega=L_{\omega,h}$.
Since 
\[\expect_\omega\norm{L_\omega}_{\text{op}}
\le \big(\expect_\omega\norm{L_\omega}_{\text{op}} \big)^{1/2q}
\le \big(\expect_\omega\trace[(L_\omega^*L_\omega)^{q}] \big)^{1/2q},
\]
it therefore suffices to show that for any positive integer $q$,
\begin{equation} \label{traceclaim}
\expect_\omega \trace (L_\omega^*L_\omega)^q
\le C_q N\cdot (Np)^{-q} \norm{h}_\infty^{2q}.
\end{equation}

Write $\vec{n}=(n_1,\cdots,n_{2q})$
where $n_j\in[-N,N]$ are arbitrary.
Define $n_{2q+1}\equiv n_1$.
All sums over $\vec{n}$ written below are understood to be
taken over all such vectors $\vec{n}\in[-N,N]^{2q}$.
We say that $\vec{n}$ is admissible
if in the vector 
\[\vec{m}=(n_2-n_1,n_2-n_3,n_4-n_3,n_4-n_5,n_6-n_5,n_6-n_7,\cdots n_{2q}-n_{2q+1}),\]
no integer appears as a coordinate exactly once.
We write $\sum_{\vec{n}}^\dagger$ to denote the sum over all admissible $\vec{n}$.

With this notation, the trace can be expanded in the form
\[
\trace (L_\omega^*L_\omega)^q
=
\sum_{\vec{n}}
H(\vec{n})
\prod_{i=1}^{2q} r^*(\omega,n_{i+1}-n_i)
\]
where $r^*(\omega,n_{i+1}-n_i)
=r(\omega,n_{i+1}-n_i)$ if 
$i$ is odd, and
$=r(\omega,n_{i}-n_{i+1})$ if $i$ is even.
Here $H(\vec{n})$ is a product of $2q$ factors of $h$,
so $\norm{H}_{\ell^\infty}\le\norm{h}_\infty^{2q}$.
Moreover,
\begin{equation} \label{expectedtrace}
\expect_\omega
\trace (L_\omega^*L_\omega)^q
=
\textstyle\sum_{\vec{n}}^\dagger
H(\vec{n})
\expect_\omega 
\prod_{i=1}^{2q} r^*(\omega,n_{i+1}-n_i),
\end{equation}
since 
\[\expect_\omega
\prod_{i=1}^{2q} r^*(\omega,n_{i+1}-n_i)=0
\]
by independence whenever $\vec{n}$ is not admissible.

If $\vec{n}$ is admissible,
then the number $K$ of pairwise distinct coordinates of $\vec{m}(n)=(n_2-n_1,n_2-n_3,\cdots)$
satisfies $K\le q$.
Fix any $K\in[1,q]$. The number of $\vec{m}=(m_1,\cdots,m_{2q})\in[-2N,2N]^{2q}$
having exactly $K$ pairwise distinct coordinates
is $\le C_q N^K$.
The number of such $\vec{m}$ possessing the additional property that 
$m_1-m_2+m_3-m_4+\cdots=0$ is of course no greater.
The number of $\vec{n}\in[-N,N]$ for which $\vec{m}(n)$ has exactly $K$
distinct coordinates is therefore $\le C_q N^{K+1}$; one additional power of
$N$ arises, because $\vec{n}$ is determined by $\vec{m}(n)$ together with $n_1$,
though not by $\vec{m}(n)$ alone.

If $\vec{m}(n)$ has $K$ pairwise distinct coordinates,
then 
\[
\expect_\omega 
\prod_{i=1}^{2q} |r^*(\omega,n_{i+1}-n_i)|
\le C_q (Np)^{-2q} p^K.
\]
Therefore the total contribution made to \eqref{expectedtrace}
by all admissible indices $\vec{n}$ having $K$ pairwise distinct coordinates is
\[
\le C_q N\cdot (Np)^{-2q}N^Kp^K
\le C_q N\cdot (Np)^{q-2q}
= C_q N\cdot (Np)^{-q}
\]
since $Np\ge 1$ and $K\le q$.
Summing over all $K$ gives \eqref{traceclaim}.
\end{proof}

\section{Reduction of degree of multilinearity}

The proof of Theorem~\ref{thm:main} will proceed by descending induction on the degree of 
multilinearity, $M$. 
In this section we set up a simple lemma which implements the inductive step.

It will be useful to reformulate and to modestly generalize the operators $T_\omega$.
Consider a scalar-valued multilinear form
\begin{equation} \label{symmetricform}
\scriptt(f_1,\cdots,f_M,\rho) = \sum_{(x,y)\in[-AN,AN]^2}
\prod_{j=1}^M f_j(L_j(x,y)) \rho(L_0(x,y))
\end{equation}
where $M\ge 2$,
$f_j:\integers\mapsto\reals$,
each $L_j:\integers\to\rationals$ is $\rationals$-linear,
$L_i$ is not a scalar multiple, over $\rationals$,
of $L_j$ if $i\ne j$,
and $\rho:\integers\to\reals$.
We operate under the convention that
$f_j(L_j(x,y))$ is to be interpreted as $0$ whenever $L_j(x,y)\in\rationals\setminus\integers$,
and likewise for $\rho(L_0(x,y))$.
Moreover, all $f_j$ and $\rho$ are supported in $[-AN,AN]$.
Here $A\ge 1$ is any positive integer, which is initially $1$ but will
increase in a controlled manner with each inductive step.
We seek to bound $\scriptt(f_1,\cdots,f_M,\rho)$
by a suitable constant times $\norm{f_1}_2\norm{f_2}_2\prod_{j>2}\norm{f_j}_\infty$.
This suitable constant will depend on $A$, in a manner which will not be
specified.
In the application, $\rho$ will depend on $\omega\in\Omega$
and will be constructed from $r(\omega,\cdot)$ in a recursive manner.

By assumption, $\integers^2\owns (x,y)\mapsto (L_1(x,y),L_2(x,y))$ is injective,
and has range equal to a lattice of rank $2$.
Make a linear ``change of variables'' $(x,y)\mapsto (u,v)=\lambda (L_1(x,y),L_2(x,y))$
where $0\ne \lambda\in\integers$ is chosen so that $\lambda L_i(x,y)\in\integers$
for $i=1,2$ for all $(x,y)\in\integers^2$.
The range of $\lambda L_1$ need not be arranged to be all of $\integers$;
set $f_1\equiv 0$ at all integers not in this range,
and likewise $f_2\equiv 0$ at all integers not in the range of $\lambda L_2$,
and for $j>2$, $f_j\equiv 0$ and $\rho\equiv 0$ at all appropriate points so that
$\scriptt(f_1,\cdots,f_M,\rho)$ may be rewritten as
\[
\scriptt(f_1,\cdots,f_M\rho) = \sum_{x,y} \rho(L_0(x,y))
f_1(x)f_2(y)\prod_{3\le j\le M}f_j(L_j(x,y));
\]
$\prod_{j\ge 3}f_j(L_j(x,y))$ is interpreted as $1$ if $M=2$.
The sum is now over $(x,y)\in [-AN,AN]^2$ for a possibly increased
value of $A$. The functions $f_j$ and linear functionals $L_j,L_0$ appearing here
are not the same as those in  \eqref{symmetricform},
but the new functionals continue to satisfy all hypotheses,
and the new functions $f_j$ have all $L^p$ norms equal to the corresponding
norms of the old functions $f_j$.

By Cauchy-Schwarz,
\begin{multline*}
|\scriptt(f_1,\cdots,f_M,\rho)|^2
\\
\le \norm{f_1}_2^2 \sum_x \sum_{y,y'} f_2(y)f_2(y')\rho(L_0(x,y))\rho(L_0(x,y'))
\prod_{j>2} f_j(L_j(x,y))f_j(L_j(x,y'))
\end{multline*}
where $x,y,y'$ are all restricted to $[-AN,AN]$.
Substitute $y'=y+z$ to reexpress the triple sum as
\[
\sum_z \sum_{x,y}
\rho^z(L(x,y))\prod_{j\ge 2}f_j^z(L_j(x,y))
\]
where $L_2(x,y)=y$,
$z\in [-2AN,A2N]$, 
\[
f_j^z(u)=f_j(u)f_j(u+L_j(0,z)),
\]
and
\[
\rho^z(u) = \rho(u)\rho(u+L_0(0,z)).
\]
Thus
\begin{gather*}
|\scriptt(f_1,\cdots,f_M,\rho)|^2
\le \norm{f_1}_2^2 
\sum_z \big|\scriptt^z(f_2^z,\cdots,f_M^z,\rho^z)\big|
\\
\intertext{where}
\scriptt^z(f_2^z,\cdots,f_M^z,\rho^z)
= \sum_{x,y} \rho^z(L(x,y))\prod_{j=2}^M f_j^z(L_j(x,y))
\end{gather*}
takes the same form as did $\scriptt$,
with the primary change
that the number of functions $f_j$ has been reduced by one.

Now $f_2^z(L_2(x,y))\equiv f_2(y)f_2(y+z)$,
so
\begin{equation} \label{eq:apparentNloss1}
\sum_z\norm{f_2^z}_2\le CN^{1/2}\norm{f_2}_2^2
\end{equation}
by Cauchy-Schwarz.
Likewise
since all functions are supported in $[-AN,AN]$,
\begin{equation} \label{eq:apparentNloss2}
\norm{f^z_3}_2\le CN^{1/2}\norm{f_3}_\infty^2.
\end{equation}

Certain values of the parameter
$z$ are exceptional, and will be treated as follows.
By Cauchy-Schwarz,
\[
\norm{f_2^z}_1\le\norm{f_2}_2^2
\]
for all $z$.
Likewise $\norm{\rho^z}_1\le\norm{\rho}_2^2$.
Therefore for any $z$,
\begin{align*}
|\scriptt^z(f_2^z,\cdots,f_M^z,\rho^z)|
&\le \norm{f_2^z}_1\prod_{j>2}\norm{f_j^z}_\infty\sup_y\sum_x|\rho^z(L_0(x,y))|
\\
&\le \norm{f_2}_2^2\prod_{j>2}\norm{f_j}_\infty^2\sup_y\sum_x|\rho^z(L_0(x,y))|
\\
&\le \norm{f_2}_2^2\prod_{j>2}\norm{f_j}_\infty^2\norm{\rho^z}_1
\\
&\le \norm{f_2}_2^2\prod_{j>2}\norm{f_j}_\infty^2\norm{\rho}_2^2.
\end{align*}

Define
\[
\norm{\scriptt(\rho)}_{\text{op}}
= \sup|\scriptt(f_1,\cdots,f_M,\rho)|
\]
where the supremum is taken over all functions
satisfying $\norm{f_j}_2\le 1$ for $j\in\{1,2\}$
and $\norm{f_j}_\infty\le 1$ for $j>2$.
Similarly
\[
\norm{\scriptt^z(\rho^z)}_{\text{op}}
= \sup|\scriptt(f_2,\cdots,f_M,\rho^z)|
\]
where the supremum is taken over all functions
satisfying $\norm{f_j}_2\le 1$ for $j\in\{2,3\}$
and $\norm{f_j}_\infty\le 1$ for $j>3$.

Write $|B|$ to denote the cardinality of a set $B$.
We have shown: 
\begin{lemma} \label{lemma:CS}
For any set $B\subset\integers$,
\begin{equation} \label{eq:TT*bound}
\norm{\scriptt(\rho)}_{\text{op}}
\le CN^{1/2}\max_{z\notin B} 
\norm{\scriptt^z(\rho^z)}_{\text{op}}^{1/2}
+|B|^{1/2}\cdot \norm{\rho}_2.
\end{equation}
\end{lemma}

\begin{remark}
It may be helpful to understand the role of the different terms here,
and the question of whether there is any essential loss when Lemma~\ref{lemma:CS} is applied.
The factors of $N^{1/2}$ in \eqref{eq:apparentNloss1} and \eqref{eq:apparentNloss2}
are natural, cannot be improved, and represent no loss. 
Indeed,
when $\rho(x)\asymp N^{-1}$ for all $x\in[-N,N]$,
$\norm{\scriptt^z(\rho^z)}_{\text{op}}=O(N^{-1})$ for all $z$, compensating exactly
for the leading factor of $N^{1/2}$ in \eqref{eq:TT*bound}; thus this factor does not
in and of itself represent any loss.
As the support of $\rho$ becomes sparser, $\norm{\rho}_\infty$ becomes large
in order to maintain the normalization $\norm{\rho}_1\asymp 1$. Since $\rho^z$ is a product
of two factors of $\rho$, $\norm{\rho^z}_\infty$ becomes larger for many values of $z$,
essentially by a factor of $N^2\norm{\rho}_\infty^{-2}$ relative to the non-sparse averaging case.
But for this loss there is also compensation; the support of $\rho^z(x)=\rho(x)\rho(x+z)$
is, on the average with respect to $z$, correspondingly smaller than the support of $\rho$. 
For natural classes of random probability measures $\rho_\omega$,
a simple back-of-the-envelope calculation gives heuristically that
$N^{1/2}\norm{\scriptt^z(\rho^z)}_{\text{op}}
=O(1)$ for typical $z,\omega$, provided that the support of $\rho$ has cardinality
$\gg N^{1/2}$. Thus one may expect to have no essential loss in applying Lemma~\ref{lemma:CS},
when dealing with random $\rho$ whose supports are not too small.

However, if the support of $\rho$ has cardinality $\ll N$, 
then $\rho_z$ will vanish identically for most values of $z$,
and Lemma~\ref{lemma:CS} must yield poor bounds.
It is this issue which leads to the restrictions on $\gamma$ in our main theorems. 
\end{remark}

In our application of Lemma~\ref{lemma:CS}, $\rho$ will take the form
\begin{equation} \label{rhoform}
\rho(x)= \rho_\omega(x)= \prod_{i\in I} r(\omega,x+z_i)
\end{equation}
where $I$ is some finite index set, and it will always be the case that
\[\text{$z_i\ne z_j$ whenever $i\ne j$.}\]
Then 
\[\rho^z(x) = \prod_{i\in I} r(\omega,x+z_i)r(\omega,x+z_i+L(0,z)).\]
Here $z\mapsto L(0,z)$ is injective.
In this situation, we define the set $B$ of exceptional values of the
parameter $z$ to be
\begin{equation} \label{Bdefn}
B=\{z: \text{ there exist } i,j\in I \text{ such that } z_i=z_j+L(0,z)\}.
\end{equation}
Then 
\[
|B|\le |I|^2,
\]
since $z\mapsto L(0,z)$ is injective.
Moreover, if $z\notin B$,
$\rho^z$ takes the same form as did $\rho$,
with the size of $I$ increased;
$B$ is defined so that the condition
\eqref{Bdefn}
is inherited from $\rho$ by $\rho^z$.

With this definition of $B$, then,
\begin{equation} \label{recursivebound}
\norm{\scriptt(\rho)}_{\text{op}}
\le CN^{1/2}\max_{z\notin B} 
\norm{\scriptt^z(\rho^z)}_{\text{op}}
+C \norm{\rho}_2
\end{equation}
where $C$ depends only on $A,|I|$.

\section{Proof of Theorem~\ref{thm:main}}

Let $\rho$ be of the form $\eqref{rhoform}$.
Then
\[
\expect_\omega(\norm{\rho_\omega}_2^2)
= \sum_x \expect\big(\prod_{i\in I} r(\omega,x+z_i)^2\big).
\]
For each $x$, the $|I|$ factors $r^2(\omega,x+z_i)$
are jointly independent since $\{z_i\}$ are distinct.
By definition,
\[
\expect(r(\cdot,x)^2)
\le C(Np)^{-2}p = CN^{-2}p^{-1}
\]
for some constant $C<\infty$.
Therefore
\[
\expect\big(\prod_{i\in I} r(\omega,x+z_i)^2\big)
\le C^{|I|} N^{-2|I|}p^{-|I|}
\]
and hence
\begin{equation}
\expect_\omega(\norm{\rho_\omega}_2^2)
\le C^{|I|} N^{1-2|I|}p^{-|I|}.
\end{equation}

A stronger result will be required.
The supremum $\sup_z$ in the next lemma is taken over
all $|I|$-tuples $z=(z_1,\cdots,z_{|I|})$ satisfying $z_i\ne z_j$ whenever $i\ne j$.

\begin{lemma} \label{lemma:secondbigq}
Let $\rho_{\omega,z}(x) = \prod_{i\in I} r(\omega,x+z_i)$
where $z_i\ne z_j$ whenever $i\ne j$.
Then for any $q<\infty$ 
there exists $C_q<\infty$ independent of $z$,
such that for every $\xi\in\torus$,
\begin{equation}
\label{eq:bigq1}
\expect_\omega(|\widehat{\rho_{\omega,z}}(\xi)|^q)
\le C_q \big(N^{-|I|+\tfrac12}p^{-|I|/2} \big)^q.
\end{equation}
Moreover, for any $\eps>0$
there exists $C_\eps<\infty$ such that
\begin{gather}
\label{eq:bigq2}
\expect_\omega\big(\sup_z\norm{\widehat{\rho_{\omega,z}}}_{L^\infty(\torus)}\big)
\le C_{\eps,|I|} 
N^{-|I|+\tfrac12+\eps}p^{-|I|/2}
\end{gather}
\end{lemma}

The proof will be given below.
By Parseval's theorem, \eqref{eq:bigq2} implies
\begin{equation}
\label{eq:bigq3}
\expect_\omega(\sup_z\norm{\rho_{\omega,z}}_2)
\le C_\eps C^{|I|}N^{-|I|+\tfrac12+\eps}p^{-|I|/2}.
\end{equation}

We are now in a position to argue by induction on the degree of multilinearity
$M$. Some additional notation is required, because the base case in the induction
depends on $M$. 
Let $\scriptt(\omega)=\scriptt_M(\omega)$ be the $M$-linear scalar-valued form to be analyzed;
thus $\rho=r$.
Define $\scriptb_{M+1}=\emptyset$, and $\scriptb_M=\{0\}\subset\integers^1$.
For $z\notin \scriptb_M$
define $\scriptt_{M-1}(\omega,z)$ to be the associated $M-1$--linear scalar form,
as discussed above.
Define $\scriptb_{M-1}$ to be the set of all $(z_1,z_2)\in\integers^2$ such that
$z_1\notin \scriptb_M$ and $z_2$ does not lie in the finite exceptional set $B$
associated to $z_1$ in the above discussion.
For $(z_1,z_2)\notin \scriptb_{M-1}$
let $\scriptt_{M-2}(\omega,(z_1,z_2))$ be the associated $M-2$--linear scalar form.
Continue, constructing 
$\scriptt_{M-k}(\omega,z)$ for $k=0,1,2,\cdots,M-2$
for (most) $z\in\integers^k$, and exceptional sets $\scriptb_{M-k}\subset\integers^{k+1}$.
For each $z\notin \scriptb_{M-k}$, $\{\zeta: (z,\zeta)\in \scriptb_{M-k-1}\}$
is a finite set whose cardinality is bounded by a constant which depends only on $k$.
By \eqref{recursivebound},
\begin{equation}
\norm{\scriptt_k(\omega,(z_1,\cdots,z_{M-k})}_{\text{op}}
\le CN^{1/2}\max_{\zeta: (z,\zeta) \notin \scriptb_{k}} 
\norm{\scriptt_{k-1}(\omega,(z_1,\cdots,z_{M-k},\zeta))}_{\text{op}}
+C \norm{\rho_{\omega,(z_1,\cdots,z_{M-k})}}_2.
\end{equation}

\begin{lemma} \label{lemma:induction}
Suppose that $p\ge N^{-1}$. Then
for any $\eps>0$ and $k\in\{2,3,\cdots,M\}$,
\begin{equation}
\expect_\omega\big(\sup_{z\notin \scriptb_{k+1}}\norm{\scriptt_k(\omega,z)}_{\text{op}}\big)
\le C_\eps N^{1+\eps}
N^{-2^{1-k}}
N^{-2^{M-k}}
p^{-2^{M-k-1}}.
\end{equation}
\end{lemma}
Specializing this conclusion to $k=M$ yields the sought-for bound.

\begin{corollary}
Provided that $p\ge N^{-1}$,
\begin{equation}
\expect_\omega\big(\norm{\scriptt(\omega)}_{\text{op}}\big)
\le C_\eps N^{\eps}
N^{-2^{1-M}}
p^{-1/2}.
\end{equation}
If $p\ge N^{-\gamma}$ and if
$\gamma < 2^{-(M-2)}$
then 
\begin{equation}
\expect_\omega\big(\norm{\scriptt(\omega)}_{\text{op}}\big)
= O(N^{-\delta})
\text{ for any $\delta < \tfrac12 (2^{-(M-2)}-\gamma)$.}
\end{equation}
\end{corollary}

\begin{proof}[Proof of Lemma~\ref{lemma:induction}]
We proceed by ascending induction on $k$.
$\scriptt_k(\omega,z)$ is associated to an index set $I$
of cardinality $|I| = 2^{M-k}$.
In the base case $k=2$,
$\scriptt_0(\omega,z)$ is the bilinear form associated
to a linear operator defined, in appropriate coordinates,
by convolution with $\rho_{\omega,z}$.
$\norm{\scriptt_0(\omega,z)}_{\text{op}}$ is simply the $L^2(\integers)\to L^2(\integers)$
operator norm of this convolution operator, which is the $L^\infty$
norm of the Fourier transform $\widehat{\rho_{\omega,z}}$.
Therefore by Lemma~\ref{lemma:secondbigq},
\begin{align*}
\expect_\omega\sup_{z\notin\scriptb_3}\norm{\scriptt_0(\omega,z)}_{\text{op}}
&\le C_\eps N^\eps N^{1/2}N^{-|I|}p^{-|I|/2}
\\
&= 
C_\eps N^\eps N^{1/2}
N^{-2^{M-2}}p^{-2^{M-3}}
\\
&= C_\eps N^{1+\eps}
N^{-2^{-1}}
N^{-2^{M-2}}p^{-2^{M-3}}
\\
&= C_\eps N^{1+\eps}
N^{-2^{1-k}}
N^{-2^{M-k}}p^{-2^{M-k-1}},
\end{align*}
which is the bound stated for $k=2$.

For the inductive step,
\begin{align*}
\expect_\omega\max_{z\notin\scriptb_{k+1}}
\norm{\scriptt_k(\omega,z)}
&\le CN^{1/2}\expect_\omega
\max_{(z,\zeta)\notin\scriptb_k} \norm{\scriptt_{k-1}(\omega,(z,\zeta))}_{\text{op}}^{1/2}
+C\expect_\omega\max_{z\notin\scriptb_{k+1}} \norm{\rho_{\omega,z}}_2
\\
&\le C_\eps N^{1/2} \Big(N^{1+\eps}
N^{-2^{1-(k-1)}}
N^{-2^{M-(k-1)}}
p^{-2^{M-(k-1)-1}}\Big)^{1/2}
+C\expect_\omega\max_z \norm{\rho_{\omega,z}}_2
\\
&\le C_\eps N^{1+\eps}
\Big(
N^{-2^{2-k}}
N^{-2^{M-k+1}}
p^{-2^{M-k}}
\Big)^{1/2}
+C\expect_\omega\max_z \norm{\rho_{\omega,z}}_2
\\
&= C_\eps N^{1+\eps}
N^{- 2^{1-k}}
N^{-2^{M-k}}
p^{-2^{M-k-1}}
+C\expect_\omega\max_z \norm{\rho_{\omega,z}}_2.
\end{align*}
The first term on the final line is of the desired form.
By Lemma~\ref{lemma:secondbigq},
\[
\expect_\omega\max_z \norm{\rho_{\omega,z}}_2
\le C_\eps N^{1/2} N^{-|I|+\eps}p^{-|I|/2}
\]
where $|I| = 2^{M-k}$. Thus
\begin{align*}
\expect_\omega\max_z \norm{\rho_{\omega,z}}_2
&\le C_\eps N^{\tfrac12+\eps} N^{-2^{M-k}}p^{-2^{M-k-1}}
\\
&= C_\eps N^{1+\eps} N^{-1/2} N^{-2^{M-k}}p^{-2^{M-k-1}}
\\
& \le C_\eps N^{1+\eps} N^{-2^{1-k}} N^{-2^{M-k}}p^{-2^{M-k-1}}
\end{align*}
since $k\ge 2$.
This completes the inductive step.
\end{proof}

\begin{proof}[Proof of Lemma~\ref{lemma:secondbigq}]
It suffices to treat the case where $q$ is an even positive integer.
Thus we may replace $q$ by $2q$.
For any $\xi\in\torus$,
\begin{align*}
\expect_\omega\big(|\widehat{\rho_{\omega,z}}(\xi)|^{2q}\big)
&= \expect_\omega
\sum_{n_1,\cdots,n_q}
\sum_{n'_1,\cdots,n'_q}
\prod_{\alpha=1}^q\rho_{\omega,z}(n_\alpha)
\prod_{\beta=1}^q\rho_{\omega,z}(n'_\beta)
e^{-i\xi(\sum_\alpha n_\alpha-\sum_\beta n_\beta)}
\\
&\le
\sum_{n_1,\cdots,n_{2q}}
\big|\expect_\omega
\prod_{\alpha=1}^{2q}\rho_{\omega,z}(n_\alpha) \big|,
\end{align*}
where $n_\alpha = n'_{\alpha-q}$ for $q>\alpha$.

For $m\in\integers$
and $\vec{n}=(n_1,\cdots,n_{2q})\in\integers^{2q}$,
define $\nu(m,\vec{n})$ to be
the number of indices $(\alpha,i)\in \{1,2,\cdots,2q\}\times I$
which satisfy $n_\alpha+z_i=m$.
Since $z_i\ne z_j$ whenever $i\ne j$, there can be at most
one such pair with a given value of $\alpha$.
For fixed $k$, the $|I|$ random variables 
$\rho_{\omega,k+z_i}$
are jointly independent and 
$\expect_\omega(\rho_{\omega,k+z_i})=0$ for each $i$.
Therefore
\[
\expect_\omega
\prod_{\alpha=1}^{2q}\rho_{\omega,z}(n_\alpha)
=0
\text{ unless  for every $m\in\integers$, } \nu(m,\vec{n})\ne 1.
\]

We say that $\vec{n}\in\integers^{2q}$ is negligible
if there exists at least one $m\in\integers$
satisfying $\nu(m,\vec{n})=1$.
The number of nonnegligible multi-indices $\vec{n}$
is $\le C_{q,|I|}(AN)^q$. 
To prove this, given $\vec{n}$,
partition the indices $1,2,\cdots,2q$ into equivalence classes, by saying that
$n_\alpha$ is equivalent to $n_\beta$ if there exist indices $i,j$
such that $n_\alpha+z_i=n_\beta+z_j$,
and forming the smallest transitive relation $\equiv$ generated
by these relations.
Each $\vec{n}$ is thereby associated to a unique equivalence relation
on $\{1,2,\cdots,2q\}$.
The number of such relations is a finite quantity, for each $q$.
Consider all $\vec{n}$ associated to a given relation, with $\scriptc$
distinct equivalence classes. $\scriptc\le q$, since each
equivalence class contains at least two elements.
If $\{\beta\}$ is a collection of indices $\alpha$,
with exactly one chosen from each equivalence class,
then for every $\alpha\notin\{\beta\}$,
$n_\alpha$ is determined from some $n_\beta$
by an equation $n_\alpha+z_j=n_\beta+z_i$.
Therefore at most $(AN)^\scriptc$ values of $\vec{n}$
remain undetermined.
Therefore there are at most $(AN)^\scriptc\le (AN)^q$
indices $\vec{n}$ associated to any given equivalence relation.

If $\vec{n}$ is not negligible then
\begin{align*}
\expect_\omega
\prod_{\alpha=1}^{2q}\rho_{\omega,z}(n_\alpha)
\le C_q(Np)^{-2q|I|}
\prod_{m: \nu(m,\vec{n})\ge 2} p
 = C_q(Np)^{-2q|I|} p^{\mu(\vec{n})}
\end{align*}
where
$\mu(\vec{n})$ is the number of $m\in\integers$
satisfying $\nu(m,\vec{n})\ge 2$.
Plainly $\mu(\vec{n})\le 2q|I|/2 = q|I|$,
so
\[
\expect_\omega
\prod_{\alpha=1}^{2q}\rho_{\omega,z}(n_\alpha)
\le  C_q(Np)^{-2q|I|} p^{q|I|}.
\]
Summing over all nonnegligible $\vec{n}$ gives
\[
\expect_\omega \big|\widehat{\rho_{\omega,z}}(\xi) \big|^{2q}
\le  C_{q,|I|}(Np)^{-2q|I|} p^{q|I|} (AN)^q
= C_{q,|I|} A^q N^{(1-2|I|)q} p^{q|I|},
\]
as was to be proved.

To derive
\eqref{eq:bigq2} is from \eqref{eq:bigq1}, 
temporarily fix any $\xi\in\torus$.
For any $q<\infty$,
\begin{multline*}
\Big(\expect_\omega\sup_z|
\widehat{\rho_{\omega,z}}(\xi)
|\Big)^q
\le
\expect_\omega\big(
\sup_z|\widehat{\rho_{\omega,z}}(\xi)|^q
\big)
\le 
\expect_\omega\big(\sum_z
|\widehat{\rho_{\omega,z}}(\xi)|^q
\big)
\\
=
\sum_z
\expect_\omega\big(
|\widehat{\rho_{\omega,z}}(\xi)|^q
\big)
\le C_{A,|I|}N^{|I|} 
\sup_z \expect_\omega\big(
|\widehat{\rho_{\omega,z}}(\xi)|^q
\big)
\\
\le C_{A,|I|}N^{|I|} 
C_q C^{|I|q}N^{-q|I|}N^{q/2}p^{-|I|q/2},
\end{multline*}
since at most $C_{A,|I|}N^{|I|}$ values of $z$ arise.
Choosing $q=|I|/\eps$ yields 
\begin{equation} \label{eq:nearlybigq2}
\expect_\omega\big(\sup_w|\widehat{\rho_{\omega,w}}(\xi)|\big)
\le C_{\eps,|I|} 
N^{-|I|+\tfrac12+\eps}p^{-|I|/2}.
\end{equation}

This is weaker than \eqref{eq:bigq3}, in which
$|\widehat{\rho_{\omega,z}}(\xi)|$
is replaced by
$\norm{\widehat{\rho_{\omega,z}}}_\infty$.
But since $\rho_{\omega,z}$
is supported on an interval $[-AN,AN]$,
by the Shannon sampling theorem
\[
\norm{\widehat{\rho_{\omega,z}}}_\infty
\le \max_{j}
|\widehat{\rho_{\omega,z}}(\xi_j)|
\]
where $\{\xi_j\}\subset\torus$ is an arithmetic progression
consisting of $KAN$ points with spacing $K^{-1}A^{-1}N^{-1}$, 
where $K$ is an absolute constant.
Since such a progression consists of $O(N)$ points,
the same reasoning used to introduce the supremum over $z$ in \eqref{eq:nearlybigq2}
also suffices to introduce the supremum over all $\xi_j$,
at the expense of another factor of $N^\eps$.
Thus \eqref{eq:bigq2} follows from \eqref{eq:nearlybigq2}.
\end{proof}

\section{Extensions}

In this section we present an extension of Theorem~\ref{thm:main},
then show how Theorem~\ref{thm:Carleson} is an almost immediate consequence
of this extension.
Finally, we show how the application to return times of random subsequences 
is deduced from Theorem~\ref{thm:Carleson}.

For $K\ge 1$ let $S_K\subset\integers^K$ be the set of all
$z=(z_1,\cdots,z_K)\in\integers^{K}$
satisfying
\begin{equation}
i\ne j\Rightarrow z_i\ne z_j.
\end{equation}
Define $\rho(\omega,z):\integers\to \reals$ by
\begin{equation}
\rho(\omega,z)(x) = 
N^{K-1} \prod_{i=1}^K r(\omega,x+z_i).
\end{equation}
Consider multilinear operators
\begin{equation}
T_{\omega,z}(f,g_1,\cdots,g_M)(x)
= \sum_y f(y)\prod_{j=1}^M g_j(L_j(x,y))
\end{equation}
where $\{L_j\}$ satisfy \dots
Define 
\[
\norm{T}_{\text{op}}
= \sup_{f,g_1,g_M} \norm{T_\omega(f,g_1,\cdots,g_M)}_2
\]
where the supremum is taken over all $f,\{g_j\}$
satisfying $\norm{f}_2\le 1 $
and $\norm{g_j}_\infty\le 1$.
Since
\[
\expect_\omega\norm{\rho(\omega,z)}_1
\asymp N^{K-1} (Np)^{-K} \cdot N\cdot p^K\equiv 1,
\]
the factor $N^{K-1}$ in the definition of $\rho(\omega,z)$ is the natural
normalization here.

The same analysis as above establishes:
\begin{theorem} \label{thm:productofr}
There exist $\gamma=\gamma(M,K)>0$
and $\delta>0$,  $C<\infty$
such that for all $p\ge N^{-\gamma}$,
\begin{equation}
\expect_\omega \sup_{z\in S_K}\norm{T_{\omega,z}}_{\text{op}}
\le CN^{-\delta}
\end{equation}
uniformly for all $N$.
\end{theorem}

\begin{proof}[Proof of Theorem~\ref{thm:Carleson}]
Introduce a function $\xi(x,z)$ so that
\[
|T_{\omega,z}(e_{\xi(x,z)} f,g_1,\cdots,g_M)(x)|
\ge \tfrac12 \sup_\xi
|T_{\omega,z}(e_{\xi} f,g_1,\cdots,g_M)(x)|
\]
for all $x\in\integers$,
where $e_\xi(x,z)$ denotes the function $x\mapsto e^{-i\xi(x,z)}$.
Now 
\begin{multline} \label{byephase}
\norm{T_{\omega,z}(e_{\xi(x,z)} f,g_1,\cdots,g_M)}_2^2
\\
=\sum_w \sum_{x,y} 
\rho(\omega,z)(L(x,y))\rho(\omega,z)(L(x,y)+L(0,w))
\,
e^{-i\xi(x,z)w} f_w(y)
\prod_j g_{j,w}(L_j(x,y))
\end{multline}
with the same notations for $f_w,g_{j,w}$ as in the beginning of the
discussion of Theorem~\ref{thm:main}.
Now in \eqref{byephase}, set $g_{0,w}(x)=e^{-i\xi(x,z)w}$.
Then $\norm{g_{0,w}}_\infty=1$.
Regarding this expression as a multilinear form in $f_w,\{g_{j,w}: 0\le j\le M\}$,
it is in a form to which Theorem~\ref{thm:productofr} applies,
yielding the desired bound.
\end{proof}

\begin{proof}[Sketch of proof of Theorem~\ref{thm:returntimes}]
For $j\in\naturals$
let $N_j$ be the number of indices $k$ such that $n_k\in [1,2^{j+1}]$.
With probability $\ge 1-e^{-c2^{\delta j}}$ for some $c,\delta>0$,
$N_j\asymp 2^{(1-\gamma)j}$.
Moreover, $N_j-N_{j-1}\asymp N_j$, with similarly high probability.

We will sketch the proof of a weaker result, namely:
\begin{equation} \label{eq:RTweaker}
\lim_{j\to\infty}
N_j^{-1} \sum_{k: n_k\in[1,2^{j+1}]}f(\tau^{n_k}(x))g(\sigma^{n_k}(y))
\text{ exists},
\end{equation}
with the same quantifiers as in Theorem~\ref{thm:returntimes};
the only distinction is that we average here only over
initial segments $n_k\in[1,2^j]$, rather than over arbitrary
initial segments $n_k\in[1,N]$.
The stronger conclusion stated in the theorem is proved by modifying the proof below,
as follows: Partition
$(2^{j-1},2^j]$ into subintervals of lengths
$2^{(1-\eta)j}$ for sufficiently small $\eta>0$,
and augment the sequence $(2^j: j\in\naturals)$ in the argument below 
by adjoining all endpoints of these subintervals.
Details are left to the reader.

For $f\in \ell^p=L^p(\integers)$ and $g\in\ell^q$ 
introduce the differences
\[
\Delta^\omega_j(f,g)(x,y) = 
(N_j-N_{j-1})^{-1} \sum_{n_k\in (2^j,2^{-j}]}f(x+n_k)g(y+n_k)
-2^{-j} \sum_{n\in(2^j,2^{j+1}]}f(x+n)g(y+n).
\]
We will show momentarily that 
\begin{equation} \label{eq:RTpfclaim}
\expect_\omega \sup_f \big\| \sup_g \norm{ \Delta_j^\omega(f,g)(x,y) }_{\ell^2_y} \big\|_{\ell^p_x}
\le C2^{-j\delta} \end{equation}
for some $\delta>0$ and $C<\infty$,
where the suprema are taken over all $f,g$ satisfying
$\norm{f}_{\ell^p}\le 1$
and
$\norm{g}_{\ell^2}\le 1$,
respectively.

For $g\in \ell^\infty$, the same bound holds with $\delta$ replaced by $0$.
A corresponding inequality with $\ell^2$ replaced by $\ell^q$ then follows
for all $q\in (2,\infty)$, by interpolation between the endpoints $q=2$ and $q=\infty$.
It then follows by transference that a corresponding maximal inequality
holds for arbitrary dynamical systems $(X,\tau)$ and $(Y,\sigma)$.
This maximal inequality, together with the almost everywhere existence of the
full averages \eqref{eq:RTknown} for $L^\infty$ functions, yields \eqref{eq:RTweaker}.

To establish \eqref{eq:RTpfclaim},
write $g$ in terms of its Fourier transform to represent
\[ \Delta_j^\omega(f,g)(x,y) = c\int_{\xi\in\torus} m_{\omega,j}(\xi,x)
e^{i\xi y}\widehat{g}(\xi) \] 
where $\torus=\reals/2\pi\integers$
and
\[ m_{\omega,j}(\xi,x) = N_j^{-1} \sum_{n_k\in(2^j,2^{j+1}]} f(x+n_k) e^{i\xi n_k}.  \]
Then
\[ \norm{ \Delta_j^\omega(f,g)(x,y)  }_{\ell^2_y} \le C\sup_\xi |m_{\omega,j}(\xi,x)| \]
for every $g$ satisfying $\norm{g}_{\ell^2_y}\le 1$, uniformly for every $\omega$.

Now
\[ \sup_\xi |m_{\omega,j}(\xi,x)|  = |T_{\omega,j}^*(f)(x)|\]
where $T_{\omega,j}^*$ is a maximal function of the type treated in Theorem~\ref{thm:Carleson},
associated to the random set $\{n_k\in (2^j,2^{j+1})$.
Recall that this set was specified using independent random selector variables $s_n(\omega)$,
such that $s_n(\omega)=1$ with probability $\asymp n^{-\gamma}$; for $n\in (2^j,2^{j+1}]$
this probability is $\asymp 2^{-j\gamma}$.
Theorem~\ref{thm:Carleson} therefore applies, and asserts that
$\expect_\omega \norm{T_{\omega,j}^*(f)}_{\ell^p_x} \lesssim 2^{-j\delta}$ for a certain $\delta>0$.
\end{proof}

\section{A Variant}

In this section we discuss the variant in which the random variables
are independent for distinct values of $(x,y)$, rather than depending only on
some scalar-valued linear functional $L(x,y)$.
Consider jointly independent random selector variables $s_\omega(x,y)$
for $(x,y)\in[-N,\cdots,N]^2$, satisfying $s_\omega(x,y)=1$ with probability
$p$, and $=0$ otherwise.
Then $\expect(\sum_x s_\omega(x,y))\asymp Np$
and $\expect(\sum_y s_\omega(x,y))\asymp Np$.
Define $r_\omega(x,y)=(Np)^{-1}\big(s_\omega(x,y)-p)$
so that
$\expect_\omega r_\omega(x,y) =0$.
Let $\scriptt_\omega$ be the associated multilinear operators.
We will sometimes write $\scriptt^{(M)}_\omega$ to indicate the degree of
multilinearity of $\scriptt_\omega$.

The factor $(Np)^{-1}$ in the definition of $r_\omega$ 
represents the natural normalization, so that
the expected value of the norm of the linear operator 
$f\mapsto \sum_y |r_\omega(x,y)| f(y)$, on $\lt([-N,N])$,
is uniformly bounded. More precisely:

\begin{lemma} \label{lemma:logloss}
For any $M\ge 2$ $A<\infty$, $\gamma_0\in[0,1)$, 
and family $\{L_j\}$ satisfying
the hypotheses of Theorem~\ref{thm:main},
there exists $C<\infty$ such that
for any $\gamma\in[0,\gamma_0]$ and
any index $i\in\{1,2,\cdots,M\}$,
for any $N\ge 1$,
\[
\expect_\omega\Big(
[
\sup_{f_1,\cdots,f_M}
\sum_{(x,y)\in [-AN,AN]^2}
|r_\omega(x,y)|\prod_{j=1}^M |f_j(L_j(x,y))| 
]^2
\Big)
\le C_{A,\gamma} \log(N)^2,
\]
where the supremum is taken over all functions satisfying
$\norm{f_i}_1\le 1$
and $\norm{f_j}_\infty\le 1$ for all $j\ne i$.
\end{lemma}

\begin{proof}[Sketch of proof]
At the expense of a factor depending on $\{L_j\}$,
we may change variables so that $L_i(x,y)=x$. Then
\[
\sum_y 
|r_\omega(x,y)|\prod_{j=1}^M |f_j(L_j(x,y))| 
\le |f_i(x)|\sum_y |r_\omega(x,y)|.
\]
An application of Chernoff's inequality (see below for a similar argument) yields
\[
\expect_\omega  \sup_{x\in[-AN,AN]}
\sum_y |r_\omega(x,y)| 
\le C\log(2+AN),
\]
and the same for the expectation of the square.
\end{proof}

$\expect_\omega  \big(\sum_y |r_\omega(x,y)|\big) $ 
is also bounded below by a strictly positive constant,
independent of $x$. Since the random variables
$\sum_y |r_\omega(x,y)|$  are jointly independent,
it is easily seen that
$\expect_\omega  \big(\sup_x\sum_y |r_\omega(x,y)|\big) $ 
is not uniformly bounded as $N\to\infty$.

Theorem~\ref{thm:morerandom} will be proved
by induction on the degree $M$ of multilinearity.
The following base result will be proved later.
\begin{lemma} \label{lemma:morerandombasecase}
For any $\eps>0$,
$\expect_\omega\norm{\scriptt_\omega^{(2)}}_{\text{op}}
\le C_\eps N^\eps N^{-(1-\gamma)/2}$.
\end{lemma}

We turn to the proof of Theorem~\ref{thm:morerandom}.
Let $M\ge 3$, and $\gamma\in[0,1)$.
By a simple interpolation,
it suffices to prove the inequality under the assumption
that each function $f_j$ equals the characteristic function
$\chi_{E_j}$ of a set $E_j\subset[-N,N]$.
We will simplify notation by writing
$\scriptt_\omega(E_1,\cdots,E_M)$ for $ \scriptt_\omega(\chi_{E_1},\cdots,\chi_{E_M})$.
Introduce the restricted weak type norm
\begin{equation}
\norm{\scriptt}_{\text{weak}}
= \sup_{E_1,\cdots,E_M}|E_1|^{-1/2}|E_2|^{-1/2}
|\scriptt(E_1,\cdots,E_M)|.
\end{equation}

Suppose now that the theorem has been proved for $M-1$.
Therefore for $E_M=[-N,N]$, 
\begin{equation} \label{inductionstep}
\expect_\omega \sup_{\{E_1,\cdots,E_{M-1}\}} |E_1|^{-1/2}|E_2|^{-1/2}
\big|\scriptt_\omega(E_1,\cdots,E_{M-1},[-N,N])\big|
\le CN^{\eps-(1-\gamma)/2}
\end{equation}
where $C$ depends on $\eps,M,\gamma$.

\begin{lemma} \label{lemma:split}
For any $\eta\in(0,1)$ and for any $\omega$,
\begin{equation}
\norm{\scriptt^{(M)}_\omega}_{\text{weak}}
\le C N^{-\eta/2}
+
\norm{\scriptt^{(M-1)}_\omega}_{\text{op}}
+ 
\textstyle{\sup^*_{\{E_m\}}}
|E_1|^{-1/2}
|E_2|^{-1/2}
\big|
\scriptt_\omega^{(M)}(E_1,\cdots,E_M)
\big|
\end{equation}
where $\sup^*_{\{E_m\}}$ denotes
the supremum over all $M$-tuples of sets $E_j$ satisfying 
\[
|E_1|\cdot|E_2|\ge N^{2-\eta}.
\]
\end{lemma}

\begin{proof}
Denote by $\one$ the constant function $\one(x)=1$ for all $x\in [-N,N]$.
Define the nonrandom averaging forms
\[\scripta(f_1,\cdots, f_M) = N^{-1}\sum_{x,y}\prod_{j=1}^M f_j(L_j(x,y)).\]
As for $\scriptt$, write 
$\scripta(E_1,\cdots,E_M)$
when each $f_j$ is the characteristic function of a set $E_j$.
Then
\[
|\scripta(E_1,\cdots,E_M)|\le C
|E_1|^{1/2}|E_2|^{1/2}
= \norm{f_1}_2\norm{f_2}_2\prod_{k>2}\norm{f_k}_\infty,
\]
but the trivial bound
\begin{equation}\label{eq:improvement}
\scripta(E_1,\cdots,E_M)\le
\scripta(E_1,E_2,\one,\cdots,\one) \le C N^{-1}|E_1|\cdot|E_2|
=C(|E_1|^{1/2}|E_2|^{1/2}/N)\cdot|E_1|^{1/2}|E_2|^{1/2}
\end{equation}
expresses a significant improvement unless $|E_1|\cdot|E_2|\asymp N$.

$\scripta(f_1,\cdots,f_M)$
never decreases if all functions are replaced by their absolute values;
nor does it decrease if some $f_j$ increases, provided that all
$f_i$ are nonnegative.
The same holds for
$(\scriptt_\omega-\scripta)(f_1,\cdots,f_M)
=(Np)^{-1}\sum_{x,y} s_\omega(x,y) \prod_j f_j(L_j(x,y))$.
Therefore if $\norm{f_j}_\infty\le 1$ for all $j\notin\{1,2\}$,
then
\begin{align*}
|\scriptt_\omega(f_1,\cdots,f_M)|
&\le |\scripta(f_1,\cdots,f_M)|
+ |(\scriptt_\omega-\scripta)(f_1,\cdots,f_M)|
\\
&\le
\scripta(|f_1|,\cdots,|f_M|)
+ (\scriptt_\omega-\scripta)(|f_1|,|f_2|,\one,\cdots,\one)
\\
&\le
2 \scripta(|f_1|,|f_2|,\one,\cdots,\one)
+ \scriptt_\omega(|f_1|,|f_2|,\one,\cdots,\one).
\end{align*}

Write $\one$ to denote  the characteristic function of $[-N,N]$,
as well as this set itself.
Let $\rho<(1-\gamma)/2$. 
Then by induction,
\[
\expect_\omega
\sup_{\{E_m: m\le M-1\}}
|E_1|^{-1/2}|E_2|^{-1/2}
|\scriptt_\omega(E_1,\cdots,E_{M-1},\one)|
\le CN^{-\rho}.
\]
Therefore by replacing $E_M$ by its complement $[-AN,AN]\setminus E_M$ if necessary,
we may assume without loss of generality that
\begin{equation} \label{eq:EMlarge}
\sum_{x,y} \prod_{j=1}^M \chi_{E_j}(L_j(x,y))
\ge \tfrac12
\sum_{x,y} \prod_{j=1}^{M-1} \chi_{E_j}(L_j(x,y)).
\end{equation}

Applying this argument to the indices $m=M-1,M-2,\dots$
in sequence, we reduce to the case where
the set 
\begin{equation}
\scripte=\{(x,y): L_j(x,y)\in E_j
\text{ for all } j\in[1,M]\}\subset E_1\times E_2
\end{equation}
satisfies
\begin{equation} \label{eq:scriptelowerbounds}
|E_1|\cdot|E_2|\le 2^M|\scripte|.
\end{equation}
\end{proof}

For any set $\scripte$ consider the random variable
\begin{equation}
X_\scripte(\omega)= Np\sum_{(x,y)\in\scripte} r_\omega(x,y)
=\sum_{(x,y)\in\scripte} \big(s_\omega(x,y)-p\big).
\end{equation}
$\expect_\omega X_\scripte(\omega)=0$.
The summands $s_\omega(x,y)-p$ are jointly independent, with values in $[-1,1]$.
$X_\scripte$ has standard deviation 
$\sigma \asymp  p^{1/2}|\scripte|^{1/2}$,
with implicit constants depending on $\gamma_0$ but not on $N$.

Chernoff's inequality \cite{taovu} asserts that
$\prob\big({|X_\scripte(\omega)|>\lambda\sigma}\big)
\le Ce^{-c\min(\lambda^2,\lambda\sigma)}$.
Set 
\begin{equation}\lambda = 
Np\cdot N^{-\rho}|\scripte|^{1/2}\sigma^{-1}
\asymp N^{1-\rho}p^{1/2}
\asymp N^{1-\rho-\gamma/2}.
\end{equation}
Then 
\begin{equation*}
\min(\lambda^2,\lambda\sigma)
= \min(N^{2-2\rho-\gamma},N^{1-\rho-\gamma}|\scripte|^{1/2})
\ge c\min(N^{2-2\rho-\gamma},N^{2-\rho-\gamma-\nu/2}).
\end{equation*}
Moreover
\[
(Np)^{-1}\lambda\sigma 
= N^{-\rho}|\scripte|^{1/2}
\le C_M N^{-\rho}|E_1|^{1/2}|E_2|^{1/2}.
\]

Consider the exceptional event
\[
\Omega^*_M(\scripte)=\{\omega\in\Omega:
|X_\scripte(\omega)|>\lambda\sigma\}.
\]
By the definition of $\lambda$, 
\begin{equation} \label{eq:pointwisebound}
|\scriptt_\omega(E_1,\cdots,E_M)| \le N^{-\rho}|E_1|^{1/2}|E_2|^{1/2}
\text{ for all $\omega\notin\Omega^*_M(\scripte)$.}
\end{equation}

Choose $\nu=2\rho$. Since $\rho<(1-\gamma)/2$,  we conclude that
\begin{equation}
\prob(\Omega^*_M(\scripte)) \le Ce^{-cN^{1+\delta}}
\end{equation}
for some $\delta>0$.

Define 
\[
\Omega^*_{M} = \Omega^*_{M-1}\bigcup \cup_\scripte \Omega^*_M(\scripte). 
\]
The total number of sets $\scripte$, of all cardinalities,
is at most $2^{CMN}$, because $\scripte$ is uniquely determined by
$E_1\times \cdots\times E_M$. 
So
\[
\prob(\cup_\scripte \Omega^*_M(\scripte))\le 
C2^{CMN}e^{-cN^{1+\delta}},
\]
and consequently $\prob(\Omega^*_M)\le CN^{-\delta}$ for another $\delta>0$.
It follows from Lemma~\ref{lemma:logloss} and H\"older's inequality
that
\begin{equation} \label{eq:exceptionalsetexpectation}
\int_{\Omega^*_M} \norm{\scriptt_\omega^{(M)}}_{\text{op}}\,d\omega \le Ce^{-cN^{1+\delta}}
\end{equation}
for some $C,c,\delta\in\reals^+$.
Since $\eta=2\rho$, \eqref{eq:exceptionalsetexpectation},
and \eqref{eq:pointwisebound}, and Lemma~\ref{lemma:split}
together give $\expect_\omega(\norm{\scriptt_\omega^{(M)}}_{\text{weak}})
\le CN^{-\rho}$.
This completes the inductive step.
\qed

\begin{proof}[Proof of Lemma~\ref{lemma:morerandombasecase}]
For any linear operator $Tf(x) = \sum_y K(x,y)f(y)$,
\begin{align*}
\norm{T}_{\text{op}}^{2K}
&\le \trace((T^*T)^K)
\\
&=
\sum
K(x_1,y_1)K(x_2,y_1)K(x_2,y_2)K(x_3,y_2)\cdots K(x_{K},y_{K})K(x_1,y_{K}).
\end{align*}
where the sum is taken over all $2K$-tuples
$(x_1,y_1,\cdots,x_{K},y_{K})$.
Apply this with $K(x,y)=K_\omega(x,y)=(Np)r_\omega(x,y)$.
Fix
$(x_1,y_1,x_2,y_2,\cdots,x_{K},y_{K})$.
Define the multiplicity of $(s,t)$
to be the number of factors $K_\omega(x_i,y_j)$
in this product for which $(x_i,y_j)=(s,t)$;
here $j=i$, or $j=i-1$, or $j=K-1$ and $i=1$.

To 
$(x_1,y_1,x_2,y_2,\cdots,x_{K},y_{K})$
is associated a nonincreasing partition of $2K$, namely the ordered tuple
of all nonzero multiplicities of elements $(s,t)$ of $\integers^2$,
written in nonincreasing order.
We denote such a partition  by
$(m_1,\cdots,m_J)$, 
where $\sum_{j=1}^J m_j=2K$.

The expectation of
$K_\omega(x_1,y_1)K_\omega(x_2,y_1)K_\omega(x_2,y_2) \cdots K_\omega(x_1,y_{K})$
vanishes unless no $(s,t)\in [1,\cdots,N]^2$ has multiplicity equal to one.
Therefore only partitions with all $m_j\ge 2$ contribute to the expectation.
The number $J$ of summands $m_j$ is then $\le K$.

\begin{lemma} \label{lemma:doti}
The number of points
$(x_1,y_1,x_2,y_2,\cdots,x_{K},y_{K})\in[1,N]^{2K}$
which give rise to any particular partition
$(m_1,\cdots,m_J)$ is
$\le  C_K N^{J+1} $.
\end{lemma}
\noindent This will be proved below.

The number of possible partitions is a function of $2K$.
\[
\expect_\omega\Big(
K_\omega(x_1,y_1)K_\omega(x_2,y_1)K_\omega(x_2,y_2) \cdots K_\omega(x_1,y_{K})
\Big)
\le C^K\prod_j p = C^Kp^J.
\]
The product $C^Kp^J\cdot N^{J+1}$ is $\le C^K N(Np)^K$ since $J\le K$.
Summing these upper bounds for expected values over all 
$(x_1,y_1,x_2,y_2,\cdots,x_{K},y_{K})$
associated to a given partition, then summing over all partititions, yields
\begin{equation}
\expect_\omega \trace((T^*T)^K)
\le C_K N (Np)^K,
\end{equation}
whence
$\scriptt_\omega = \scriptt_\omega^{(2)}$ satisfies
$\expect_\omega\norm{\scriptt_\omega}_{\text{op}}
\le C_K(Np)^{-1}N^{1/K}(Np)^{1/2}
= C_K N^{1/K}(Np)^{-1/2}$.
Since $K$ may be taken to be arbitrarily large, this establishes
Lemma~\ref{lemma:morerandombasecase}.
\end{proof}

\begin{proof}[Proof of Lemma~\ref{lemma:doti}]
Write $x_{K+1}=x_1$ to facilitate the discussion. 
If $J=1$ then $(x_1,y_1)=(x_i,y_i)$ for all $2\le i\le K$,
and there are $N^2=N^{J+1}$ possible values of $(x_1,y_1)$.
If $J>1$,
set $z_1=(x_1,y_1)$, $z_2=(x_2,y_1)$, 
$z_3=(x_2,y_2)$, $z_4=(x_3,y_2)$, \dots $z_{2K-1}=(x_K,y_K)$,
$z_{2K}=(x_1,y_K)$.
To the partition $(m_1,m_2,\cdots,m_J)$ of $2K$
we associate all possible equivalence relations on
$\{z_k: 1\le k\le 2K\}$
such that there are $J$ equivalence classes, with
$m_1,m_2,\cdots,m_J$ elements.
Such an equivalence relation  is said to be {\em feasible}
if there exist values of the $z_k\in [1,N]$
such that $z_l=z_k$ if and only if
$z_l,z_k$ belong to the same equivalence class.
The number of equivalence relations is a function of $K$ alone,
so it suffices to bound the 
number of points
$(x_1,y_1,x_2,y_2,\cdots,x_{K},y_{K})\in[1,N]^{2K}$
which give rise to one equivalence relation.

Consider any feasible equivalence relation associated to the partition
$(m_1,\cdots,m_J)$.
Choose some equivalence class with $m_1$ elements $z_k$.
Choose two coordinates, $x_i$ and $y_i$ or
$x_{i+1}$ and $y_i$, which determine all $z_k$ in this class.
These coordinates are said to be free, while
any $x_l$ or $y_l$ which is one of the two coordinates
of some $z_k$ in this class, is said to be bound.
Thus the first equivalence class accounts for exactly two free coordinates.

There must exist either $z_k=(x_i,y_i)$ in this class such that
$(x_{i+1},y_i)$ does not belong to this class,
or $z_k=(x_{i+1},y_i)$ such that $(x_{i+1},y_i)$ does not belong to this class;
otherwise the class would include every $z_k$, which is impossible since $J>1$.
In the first case, $(x_{i+1},y_i)$ belongs to a second equivalence class.
The coordinates of any other $z_k$ in this second class are determined
by $x_{i+1},y_i$. $y_i$ is a coordinate of some element of the first
class.
$x_{i+1}$ cannot be a coordinate of some element of the first
class, since $(x_{i+1},y_i)$ would belong to that class.
Designate $x_{i+1}$ to be a free coordinate, all coordinates
of all other $z_l$ in the second class are determined
by $x_{i+1}$ and $y_i$, hence by $x_{i+1}$ together with 
the two free coordinates associated to the first class.
Thus three free coordinates (together with the equivalence relation itself)
are required to determine
all coordinates of all points in the union of the first two classes.
Repeating this reasoning, we obtain if $J>2$ a third class and one additional
free coordinate, and so on. Proceeding through all $J$ classes,
a total of $J+1$ free coordinates are obtained.
Each of these coordinates can take on $N$ values, so in all there are
$N^{J+1}$ possible points associated to an individual equivalence relation
associated to a partition with $J$ elements.
\end{proof}


\begin{thebibliography}{25}


\bibitem{austinpleasant1}
T.~Austin,
Pleasant extensions retaining algebraic structure. 
arXiv:0905.0518 

\bibitem{bourgainreturn}
J.~Bourgain,
{\em Temps de retour pour les syst\`emes dynamiques}, 
C. R. Acad. Sci. Paris S\'er. I Math. 306 (1988), 483 -- 485. 

\bibitem{bourgainreturn2}
\bysame,
{\em Pointwise ergodic theorems for arithmetic sets, with appendix ``Appendix on return-time sequences''}
by J. Bourgain, H. Furstenberg, Y. Katznelson, and D. S. Ornstein, Inst. Hautes 'Etudes Sci. Publ. Math. 69 (1989), 5 -- 45. 

\bibitem{dltt}
C.~Demeter, M.~Lacey, T.~Tao, and C.~Thiele,
{\em Breaking the duality in the return times theorem},
Duke Math. J. 143 (2008), no. 2, 281--355. 


\bibitem{hostkra}
B.~Host and B.~Kra,
Nonconventional ergodic averages and nilmanifolds. (English summary) 
Ann. of Math. (2) 161 (2005), no. 1, 397--488. 


\bibitem{tao}
T.~Tao, 
Norm convergence of multiple ergodic averages for commuting transformations. (English summary) 
Ergodic Theory Dynam. Systems 28 (2008), no. 2, 657--688. 

\bibitem{taovu}
T.~Tao and V.~Vu, {\em Additive Combinatorics},
Cambridge University Press, 2010.

\end{thebibliography}
 \end{document}